\documentclass[11pt]{amsart}
\usepackage{geometry,ulem}
\usepackage{amscd,amssymb,verbatim,xcolor,mathrsfs}
\usepackage{longtable, multirow}
\usepackage{soul,cancel}
\usepackage{apptools}
\usepackage{xcolor}
\usepackage{tikz} 
\usepackage{longtable}
\usetikzlibrary {arrows.meta,bending,positioning}
\usepackage{dynkin-diagrams}
\date{\today}

\newcommand\la{{\lambda}}

\newcommand\pmat{\begin{pmatrix}}
\newcommand\epmat{\end{pmatrix}}

\newcommand\bC{{\mathbb C}}

\newcommand\fra{{\mathfrak a}}
\newcommand\frb{{\mathfrak b}}

\newcommand\frg{{\mathfrak g}}
\newcommand\frh{{\mathfrak h}}
\newcommand\frk{{\mathfrak k}}

\newcommand\frp{{\mathfrak p}}
\newcommand\frq{{\mathfrak q}}

\newcommand\frt{{\mathfrak t}}

\newcommand\bbC{{\mathbb C}}

\newtheorem{theorem}{Theorem}[section]

\newtheorem{conjecture}[theorem]{Conjecture}
\newtheorem{definition}[theorem]{Definition}
\newtheorem{example}[theorem]{Example}
\newtheorem{lemma}[theorem]{Lemma}
\newtheorem{proposition}[theorem]{Proposition}

\begin{document}
\title[FPP conjecture for complex Lie groups]{Vogan's FPP conjecture for complex Lie groups}
\author{Chao-Ping Dong}
\author{Kayue Daniel Wong}

\address[Dong]{School of Mathematical Sciences, Soochow University, Suzhou 215006, P.R. China}
\email{chaopindong@163.com}

\address[Wong]{School of Science and Engineering, The Chinese University of Hong Kong, Shenzhen,
Guangdong 518172, P. R. China}
\email{kayue.wong@gmail.com}

\dedicatory{{Dedicated to Professor David Vogan on his $70^{th}$ birthday}}

\begin{abstract}
In this paper, we give a proof of Vogan's fundamental parallelepiped (FPP) conjecture for complex simple Lie groups, resulting in a reduction step in the classification of irreducible unitary representations
for these groups.
\end{abstract}

\maketitle
\setcounter{tocdepth}{1}

\section{Introduction}\label{sec:intro}

In representation theory of Lie groups, it is a long-standing problem to classify the unitary dual. Before the end of the twentieth century, Salamanca-Riba and Vogan proposed a conjecture aiming to reduce the classification problem to that of irreducible unitary representations containing unitarily small (\textit{u-small} for short) $K$-types in \cite{SV98}. This conjectural reduction is amazing since u-small $K$-types are those $K$-types whose highest weights belong to the u-small convex hull. In particular, they are finitely many. The recent paper \cite{PP24} of  Pand\v{z}i\'{c} and Prli\'{c} reports the long-term study of their team on the conjecture. It may serve as an evidence for the difficulty of the conjecture.

\smallskip
In 2023, Vogan sharpened their original conjecture in \cite{SV98} to the fundamental parallelepiped (\textit{FPP} for short) conjecture as follows:

\begin{conjecture} \label{conj-main}
Let $G$ be a real reductive Lie group in the Harish-Chandra class. Let $\pi$ be any fully supported irreducible $(\frg, K)$-module with real infinitesimal character $\Lambda$. If $\langle \Lambda, \beta^{\vee}\rangle >1$ for any simple root $\beta$, then $\pi$ is non-unitary.    
\end{conjecture}

Here a $(\mathfrak{g},K)$-module $\pi$ is fully supported means that $\pi$ cannot be realized as a cohomologically induced module from any proper $\theta$-stable parabolic subalgebra of $\mathfrak{g}$ in the weakly good range. Under this perspective, along with the fact (from \cite{KV95} for instance) that cohomological induction in weakly good range (if non-zero) preserves unitarity and non-unitarity, the above conjecture gives a very effective reduction on the classification of the unitary dual $\widehat{G}$. Namely, one can only focus on the classification problem on representations with `small' infinitesimal characters. We also remark that a discrete version of Vogan's FPP conjecture appeared in Conjecture 1.4 of \cite{DDH22}, which is originated from the study of Dirac series. 

The FPP conjecture is proved for $G = U(p,q)$ in by the second named author in \cite{W24}. 
In this paper, we will present a proof of FPP conjecture for all complex simple Lie groups. Although the unitary dual is known for classical groups and $G_2$ (\cite{D79}, \cite{V86}, \cite{B89}, \cite{V94}), we present the proofs of these cases for completeness. Moreover, this can serve as a demonstration of certain methods employed in \cite{B89} for classical groups, and will shed some light on proving the conjecture for all real reductive groups in general.

\section{Preliminaries}

From now on, if not stated otherwise, we let $G$ be a connected complex simple Lie group and  view it as a real Lie group.

\subsection{Complex groups} \label{sec-cx}
Fix a Cartan subgroup $H$ of $G$ and a Cartan
involution $\theta$ of $G$. Write $K:=G^{\theta}$, which is a maximal
compact subgroup of $G$. Denote by $\frg_0=\frk_0\oplus\frp_0$ the corresponding Cartan
decomposition of the Lie algebra $\frg_0 ={\rm Lie} (G)$, and $\frh_0 = \frt_0 \oplus \fra_0$, where $\mathfrak{h}_0={\rm Lie} (H)$. We remove the subscript from a Lie algebra to denote its 
complexification. The following identifications are adopted throughout this paper:
\begin{equation}\label{identifications}
\frg\cong \frg_{0} \oplus
\frg_0, \quad
\frh\cong \frh_{0} \oplus
\frh_0, \quad \frt\cong \{(x,-x) : x\in
\frh_{0} \}, \quad \fra \cong\{(x, x) : x\in
\frh_{0} \}.
\end{equation}

Fix a Borel subgroup $B$ of $G$
containing $H$. Put $\Delta^{+}(\frg_0, \frh_0)=\Delta(\frb_0, \frh_0)
$ and let $W$ be the Weyl group $W = W(\frg_0, \frh_0)$. 
We may and we will identify a $K$-type $V_{\frg}(\eta)$ with its highest weight $\eta \in \frt^* \cong \mathfrak{h}_0^*$.

Let $(\lambda_{L}, \lambda_{R})\in \frh_0^{*}\times
\frh_0^{*}$ be such that $\lambda_{L}-\lambda_{R}$ is
a weight of a finite dimensional holomorphic representation of $G$.
 We view $(\lambda_L, \lambda_R)$ as a real-linear functional on $\frh$ by \eqref{identifications}, and write $\bbC_{(\lambda_L, \lambda_R)}$ as the character of $H$ with differential $(\lambda_L, \lambda_R)$. By \eqref{identifications} again, we have
$$
\bbC_{(\lambda_L, \lambda_R)}|_{T}= \bbC_{\mu}:=\bbC_{\lambda_L-\lambda_R}, \quad \bbC_{(\lambda_L, \lambda_R)}|_{A}=\bbC_{\nu}:=\bbC_{\lambda_L+\lambda_R}.
$$
Extend $\bbC_{(\lambda_L, \lambda_R)}$ to a character of $B$, and put 
$$X(\lambda_{L}, \lambda_{R}):=K\mbox{-finite part of Ind}_{B}^{G}(\bbC_{\mu} \otimes \bbC_{\nu} \otimes {\bf 1}),$$
where the parabolic induction is the normalized so that its infinitesimal character is equal to the $W\times W$ orbit of $(\lambda_L, \lambda_R)$. Using Frobenius reciprocity, one sees that the $K$-type with extremal weight $\mu$ occurs with multiplicity one in $X(\lambda_{L}, \lambda_{R})$. Let $\{\xi\}$ be the unique dominant weight to which $\xi$ is conjugate under the action of $W$, and let $J(\lambda_L,\lambda_R)$ be the unique subquotient of $X(\lambda_{L}, \lambda_{R})$ containing the $K$-type $V_{\frg}(\{\mu\})$.

\begin{theorem}\label{thm-Zh} {\rm (Zhelobenko \cite{Zh74})}
In the above setting, we have that
\begin{itemize}
\item[a)] Every irreducible admissible ($\frg$, $K$)-module is of the form $J(\lambda_L,\lambda_R)$.
\item[b)] Two such modules $J(\lambda_L,\lambda_R)$ and
$J(\lambda_L^{\prime},\lambda_R^{\prime})$ are equivalent if and
only if there exists $w\in W$ such that
$w\lambda_L=\lambda_L^{\prime}$ and $w\lambda_R=\lambda_R^{\prime}$.
\item[c)] $J(\lambda_L, \lambda_R)$ admits a non-degenerate Hermitian form if and only if there exists
$w\in W$ such that $w\mu =\mu , w\nu = -\overline{\nu}$.
\end{itemize}
\end{theorem}

In this manuscript, we only consider irreducible Hermitian $(\frg,K)$-modules with real infinitesimal characters. In such a case, Theorem \ref{thm-Zh}(b)-(c) imply that one can only consider the irreducible modules 
\begin{equation} \label{eq-hermirr}
J(\lambda, -w\lambda) := J\left(\frac{1}{2}\eta + \nu, -\frac{1}{2}\eta + \nu\right)
\end{equation}
where $\eta$ is a dominant weight, and $w \in W$ is chosen such that $w\eta = \eta$, $w\nu = -\nu$. 
In particular, its lowest $K$-type is $V_{\frg}(\eta)$ with infinitesimal character equal to (the $W \times W$-orbit) of  $(\lambda,-\lambda)$.

\subsection{Bottom Layer $K$-types} \label{sec-bottom}
For the rest of the manuscript, we fix the following notations: Let $\pi = J(\lambda, -w\lambda)$ be an irreducible Hermitian module with real infinitesimal character as in Equation \eqref{eq-hermirr}, whose lowest $K$-type is $V_{\mathfrak{g}}(\eta)$ with
\begin{equation} \label{eq-etal}
\eta = [k_1, k_2, \dots, k_n] := k_1 \varpi_1 + k_2 \varpi_2 + \dots k_n \varpi_n,
\end{equation}
where $\varpi_1, \dots, \varpi_n$ are the fundamental weights of $K$. Let $M_f$ be the Levi subgroup of $G$ determined by the nodes 
\begin{equation} \label{eq-m}
I(M_f) := \{i\ |\ k_i = \langle \eta,\beta_i^{\vee} \rangle = 0\} \quad \quad (\beta_i\ \textrm{simple root}).
\end{equation}
We call $M_f$ the {\bf Levi subgroup corresponding to $\pi$}. More generally, for any Levi subgroup $M$ of $G$, we will specify $M$ by a subset of the nodes in the Dynkin diagram of $G$.

\smallskip
For any Levi subgroup $M$ containing $M_f$, $\eta$ defines a dominant weight of both $M$ and $G$ (for instance, if $M = M_f$, then $V_{\mathfrak{m}_f}(\eta)$ is the trivial $M_f \cap K$-type tensored with a character in the center of $M_f$). Then there exists a unique irreducible $(\mathfrak{m}, M\cap K)$-module $\pi_{M}$ with lowest $(M \cap K)$-type $V_{\mathfrak{m}}(\eta)$ such that the induced module
\begin{equation} \label{eq-bl}
\mathrm{Ind}_{MN}^G\left(\pi_{M} \otimes {\bf 1} \right)
\end{equation}
has $\pi$ being its lowest $K$-type
subquotient. 

\smallskip
The following definition is essential in describing the relationships between the (non)-unitarity of $\pi_{M}$ and $\pi$.

\begin{definition} \label{def-bottomlayer}
Let $\pi$ be an irreducible, Hermitian $(\frg,K)$-module with $M_{\pi}$ being the Levi subgroup attached to $\pi$. For any $M \supset M_f$, consider the induced module \eqref{eq-bl} corresponding to $\pi$. We say a $(M \cap K)$-type $V_{\mathfrak{m}}(\gamma)$ appearing in $\pi_{M}$ {\bf bottom layer} if $\gamma$ also defines a dominant weight of $K$. 
\end{definition}

For instance, if the lowest $K$-type of $\pi$ is $V_{\mathfrak{g}}(\eta)$, then for any $M \supset M_f$, the lowest $(M \cap K)$-type $V_{\mathfrak{m}}(\eta)$ in $\pi_{M}$ is always bottom layer. 

\smallskip
The importance of bottom layer $K$-types on the unitarity of $\pi$ is given by the following:
\begin{theorem}[Speh-Vogan] \label{thm-SV}
Retain the notations in Definition \ref{def-bottomlayer}. Suppose the $(M \cap K)$-type $V_{\mathfrak{m}}(\gamma)$ of $\pi_M$ 
is bottom layer, then the multiplicities and signatures of $V_{\mathfrak{g}}(\gamma)$ in $\pi$ and $V_{\mathfrak{m}}(\gamma)$ in $\pi_{M}$ are equal.
\end{theorem}

\begin{example} \label{eg-bottom}
We will apply Theorem \ref{thm-SV} extensively for $M = M_f$ (so that $V_{\mathfrak{m}}(\eta)$ is a one-dimensional $(M \cap K)$-type by the above discussions) and consider 
$$\gamma = \eta + \delta,$$ 
where $\delta$ the highest weight of an adjoint $(M \cap K)$-type of $\pi_{M}$. We will present an example to see how one can check whether $\gamma$ is bottom layer, so that one can apply Theorem \ref{thm-SV} to check the (non)-unitarity of $\pi$ by that of $\pi_M$.

\medskip
Firstly, we list the highest weights $\delta$ in terms of sum of simple roots for all simple Lie groups. We will use the same numbering of nodes of the Dynkin diagrams for the rest of the manuscript.
\begin{center}
\begin{longtable}{|c|c|c|c|}
\caption{Highest weights of adjoint representations}
\label{table-adjoint}
\\
\hline
Type & Dynkin diagram & Highest weight $\delta$ of adjoint representation \\
\hline \hline
$A_p$ & 
\begin{tikzpicture}
\dynkin[labels={\beta_1,\beta_2,\beta_{p-1},\beta_p},text style/.style={scale=1},edge length=1cm]A{}
\end{tikzpicture} &
$\beta_1 + \beta_2 + \dots + \beta_{p-1} + \beta_p$\\

\hline
$B_p$ & 
\begin{tikzpicture}
\dynkin[labels={\beta_1,\beta_2,\beta_{p-2},\beta_{p-1},\beta_p},text style/.style={scale=1},edge length=1cm]B{}
\end{tikzpicture} &
$\beta_1 + 2\beta_2 + \dots + 2\beta_{p-1} + 2\beta_p$\\

\hline
$C_p$ & 
\begin{tikzpicture}
\dynkin[labels={\beta_1,\beta_2,\beta_{p-2},\beta_{p-1},\beta_p},text style/.style={scale=1},edge length=1cm]C{}
\end{tikzpicture} &
$2\beta_1 + 2\beta_2 + \dots + 2\beta_{p-1} + \beta_p$\\

\hline
$D_p$ & 
\begin{tikzpicture}
\dynkin[labels={\beta_1,\beta_2,\beta_{p-3},\beta_{p-2},\beta_{p-1},\beta_p},text style/.style={scale=1},edge length=1cm,label directions={,,,right,,}]D{}
\end{tikzpicture} &
$\beta_1 + 2\beta_2 + \dots + 2\beta_{p-2} + \beta_{p-1} + \beta_p$\\

\hline
$E_6$ & 
\begin{tikzpicture}
\draw
    (0,0) node[circle,fill=black,inner sep=0pt,minimum size=3pt,label=below:{$\beta_1$}] {}
 -- (1,0) node[circle,fill=black,inner sep=0pt,minimum size=3pt,label=below:{$\beta_3$}] {}
--   (2,0) node[circle,fill=black,inner sep=0pt,minimum size=3pt,label=below:{$\beta_4$}] {}
 -- (3,0) node[circle,fill=black,inner sep=0pt,minimum size=3pt,label=below:{$\beta_5$}] {}
-- (4,0) node[circle,fill=black,inner sep=0pt,minimum size=3pt,label=below:{$\beta_6$}] {};

\draw
 (2,0.7) node[circle,fill=black,inner sep=0pt,minimum size=3pt,label=above:{$\beta_2$}] {}
--   (2,0) node {};
\end{tikzpicture} &
$\beta_1 + 2\beta_2 + 2\beta_3 + 3\beta_4 + 2\beta_5 + \beta_6$\\

\hline
$E_7$ & 
\begin{tikzpicture}
\draw
    (0,0) node[circle,fill=black,inner sep=0pt,minimum size=3pt,label=below:{$\beta_1$}] {}
 -- (1,0) node[circle,fill=black,inner sep=0pt,minimum size=3pt,label=below:{$\beta_3$}] {}
--   (2,0) node[circle,fill=black,inner sep=0pt,minimum size=3pt,label=below:{$\beta_4$}] {}
 -- (3,0) node[circle,fill=black,inner sep=0pt,minimum size=3pt,label=below:{$\beta_5$}] {}
-- (4,0) node[circle,fill=black,inner sep=0pt,minimum size=3pt,label=below:{$\beta_6$}] {}
-- (5,0) node[circle,fill=black,inner sep=0pt,minimum size=3pt,label=below:{$\beta_7$}] {};

\draw
 (2,0.7) node[circle,fill=black,inner sep=0pt,minimum size=3pt,label=above:{$\beta_2$}] {}
--   (2,0) node {};
\end{tikzpicture}  &
$2\beta_1 + 2\beta_2 + 3\beta_3 + 4\beta_4 + 3\beta_5 + 2\beta_6 + \beta_7$\\
\hline
$E_8$ & 
\begin{tikzpicture}
\draw
    (0,0) node[circle,fill=black,inner sep=0pt,minimum size=3pt,label=below:{$\beta_1$}] {}
 -- (1,0) node[circle,fill=black,inner sep=0pt,minimum size=3pt,label=below:{$\beta_3$}] {}
--   (2,0) node[circle,fill=black,inner sep=0pt,minimum size=3pt,label=below:{$\beta_4$}] {}
 -- (3,0) node[circle,fill=black,inner sep=0pt,minimum size=3pt,label=below:{$\beta_5$}] {}
-- (4,0) node[circle,fill=black,inner sep=0pt,minimum size=3pt,label=below:{$\beta_6$}] {}
-- (5,0) node[circle,fill=black,inner sep=0pt,minimum size=3pt,label=below:{$\beta_7$}] {}
-- (6,0) node[circle,fill=black,inner sep=0pt,minimum size=3pt,label=below:{$\beta_8$}] {};

\draw
 (2,0.7) node[circle,fill=black,inner sep=0pt,minimum size=3pt,label=above:{$\beta_2$}] {}
--   (2,0) node {};
\end{tikzpicture}  &
$2\beta_1 + 3\beta_2 + 4\beta_3 + 6\beta_4 + 5\beta_5 + 4\beta_6 + 3\beta_7 + 2\beta_8$\\
\hline
$F_4$ & 
\begin{tikzpicture}
\dynkin[labels={\beta_1,\beta_2,\beta_3,\beta_4},text style/.style={scale=1},edge length=1cm]F4
\end{tikzpicture} &
$2\beta_1 + 3\beta_2 + 4\beta_3 + 2\beta_4$\\
\hline
$G_2$ & 
\begin{tikzpicture}
\dynkin[reverse arrows, labels={\beta_1,\beta_2},text style/.style={scale=1},edge length=1cm]G2
\end{tikzpicture} &
$3\beta_1 + 2\beta_2$\\
\hline
\end{longtable}
\end{center}

Here is an example on checking whether a certain adjoint $(M \cap K)$-type is bottom layer or not -  let $G$ be of type $F_4$, whose Dynkin diagram is as given above. Suppose $\pi$ is a $(\mathfrak{g},K)$-module with lowest $K$-type $\eta = [0,0,1,0]$. Then $M_f$ is of type $A_2 \times A_1$ with $I(M_f) = \{1,2,4\}$. 




Consider $M = M_f$. Then there are two simple components of $M$ of type $A_2$ (corresponding to nodes $\{1,2\}$) and type $A_1$ (corresponding to the node $\{4\}$) respectively.
For the $A_2$ component, its adjoint representation has highest weight $\delta = \beta_1 + \beta_2$. Then one can check that 
\begin{align*}
\langle \eta + \delta , \beta_1^{\vee} \rangle &= 0 + \langle \beta_1 + \beta_2, \beta_1^{\vee} \rangle = 0 +2-1 = 1\\
\langle \eta + \delta , \beta_2^{\vee} \rangle &= 0 + \langle \beta_1 + \beta_2, \beta_2^{\vee} \rangle = 0 -1 +2 = 1\\
\langle \eta + \delta , \beta_3^{\vee} \rangle &= 1 + \langle \beta_1 + \beta_2, \beta_3^{\vee} \rangle = 1 + 0 - 2 = -1 \\
\langle \eta + \delta , \beta_4^{\vee} \rangle &= 0 + \langle \beta_1 + \beta_2, \beta_4^{\vee} \rangle = 0 + 0 + 0 = 0
\end{align*}
Since the third coordinate is $-1 < 0$, the adjoint $(M \cap K)$-type $V_{\mathfrak{m}}(\eta+\delta)$ is not bottom layer.

On the other hand, the adjoint representation of the $A_1$ component has highest weight $\delta' = \beta_4$. Therefore, one has
\begin{align*}
\langle \eta + \delta' , \beta_1^{\vee} \rangle &= 0 + \langle \beta_4, \beta_1^{\vee} \rangle = 0 + 0 = 0\\
\langle \eta + \delta' , \beta_2^{\vee} \rangle &= 0 + \langle \beta_4, \beta_2^{\vee} \rangle = 0 + 0 = 0\\
\langle \eta + \delta' , \beta_3^{\vee} \rangle &= 1 + \langle \beta_4, \beta_3^{\vee} \rangle = 1  - 1 = 0 \\
\langle \eta + \delta' , \beta_4^{\vee} \rangle &= 0 + \langle \beta_4, \beta_4^{\vee} \rangle = 0 + 2 = 2
\end{align*}
Since all the above coordinates are non-negative, $V_{\mathfrak{m}}(\eta+\delta')$ is a bottom layer $(M \cap K)$-type.
\end{example}

\subsection{Support of a module}
\begin{definition} \label{def-support}
Let $G$ be a complex reductive Lie group, and $\pi = J(\Lambda,-x\Lambda)$ be an irreducible $(\frg,K)$-module such that $\Lambda$ is dominant. Then the {\bf support} of $\pi$ is defined by
$$\mathrm{supp}(\pi) := \{i\ |\ \text{the reduced expression of}\ x\ \text{contains the simple reflection}\ s_i\}.$$
We say $\pi$ is {\bf fully supported} $\mathrm{supp}(\pi) = \{1,2,\dots, \mathrm{rank}(G)\}$.
\end{definition}

Essentially, the above definition is a special case of a result of Vogan \cite{V84}, rephrased by Paul on page 23 of the slides \cite{Pau17} in the language of \texttt{atlas}. The point is that whenever the Weyl group element $x$ does not contain certain simple reflection $s_i$, then the representation $\pi$ can be chomologically induced from the Levi factor of a proper $\theta$-stable parabolic subalgebra of $G$ within the good range. In the case of complex groups, where real parabolic induction and cohomological induction are essentially the same (modulo a unitary character), if $\pi$ is not fully supported, then the proper Levi subgroup $L$ where $\pi$ is induced from is determined by $I(L) = \mathrm{supp}(\pi)$.

Vogan's result had been extensively used in the study of Dirac series. For instance, Paul's interpretation was quoted as Theorem 2.4 of \cite{DDY20}, which initiated the classification of Dirac series of real exceptional Lie groups.

One is interested in the case of Hermitian modules with real infinitesimal character, where the irreducible modules are all of the form 
$$J(\lambda,-w\lambda) = J(\frac{1}{2}\eta+\nu,-\frac{1}{2}\eta+\nu)$$ 
as described in Section \ref{sec-cx}. 

Let $\Lambda:=\{\lambda\}$ be the dominant weight to which $\lambda$ is conjugate under the Weyl group action. We may and we will arrive at $\Lambda$ from $\lambda$ as follows: If $\lambda$ is dominant, we are done; otherwise select the smallest index $k_1$ such that $\lambda_{k_1}<0$ and apply the simple reflection $s_{k_1}$. We operate on $s_{k_1}\lambda$ similarly. By Theorem 4.5 of \cite{E95}, we will stop after finitely many steps and get $\Lambda=s_{k_p}\cdots s_{k_1}\lambda$. Moreover, the expression $s_{k_p}\cdots s_{k_1}$ is reduced. For convenience, we call this way of conjugating $\lambda$ to $\Lambda$ the \textit{first negative index algorithm}.

Note that the condition $w\eta = \eta$ implies that $w$ is in the Weyl subgroup $W(M_f)$ of $W$.
In other words, the reduced expression of $w$ consists only of $s_i$ with $i \in I(M_f)$.
Then we have the following:
\begin{proposition} \label{prop-full}
    Let $\pi = J(\lambda,-w\lambda)$ be as above. Let $s\in W$ be the element $s\lambda=\Lambda$ which is given by the first negative index algorithm. If $\pi$ is fully supported, then any reduced expression of $s$ must contain all simple reflections $s_j$ with $j \notin I(M_f)$. 
\end{proposition}
\begin{proof}
 Suppose on the contrary that one (thus all) reduced decomposition of $s$ does not contain the simple reflection $s_{j_0}$ for some $j_0 \notin I(M_f)$. Then by Theorem \ref{thm-Zh}(b), 
$$
\pi = J(s\lambda,-sw\lambda) = J(\Lambda,-sws^{-1}\Lambda).
$$
Since any reduced expression of $w$ does not contain any $s_k$ with $k \notin I(M_f)$, it follows that any reduced expression of $sws^{-1}$ does not contain $s_{j_0}$.  This contradicts to the assumption that $\pi$ is fully supported.
\end{proof}

We are now in the position to phrase a stronger version of Conjecture \ref{conj-FPP} for complex groups. To begin with, we say $\Lambda \in \mathfrak{h}^*$ {\bf lies inside FPP} if 
$$\langle \Lambda, \beta_i^{\vee} \rangle \leq 1$$ 
for all simple roots $\beta_i$. 
\begin{conjecture} \label{conj-FPP} Let $G$ be a complex simple Lie group, and $\pi$ be an irreducible, fully supported, Hermitian $(\mathfrak{g},K)$-module with real infinitesimal character $(\Lambda,-\Lambda)$ such that $\Lambda$ is dominant. 
Suppose $\Lambda$ does not lie inside FPP, then $\pi$ is not unitary up to level $\mathfrak{p}$. \end{conjecture}

Recall from \cite{W24} that a Hermitian module $\pi$ with lowest $K$-type $V_{\mathfrak{g}}(\tau)$ is {\it not unitary up to level $\mathfrak{p}$} if 
there exists a $K$-type $V_{\mathfrak{g}}(\gamma)$ in the decomposition of $V_{\mathfrak{g}}(\tau) \otimes \mathfrak{p}$ such that $\pi|_K$ has a copy of $V_{\mathfrak{g}}(\gamma)$ whose signature is opposite to that of $V_{\mathfrak{g}}(\tau)$.
Also, note that in the case when $G$ is a complex simple group, then $\mathfrak{p} \cong \mathfrak{g}_0$ is the adjoint $K$-type.

\subsection{General strategy} \label{sec-strategy}
Our strategy of proving Conjecture \ref{conj-FPP} for complex simple Lie groups $G$ is as follows:
\begin{enumerate}
    \item[(I)] For each complex simple Lie group $G$, we define the notion of {\bf complex basic (cx-basic)} $K$-types (Definition \ref{def-cxfund}). In particular, if the Dynkin diagram of $G$ is simply laced, then the only cx-basic $K$-type is the trivial one.

    \item[(II)] Suppose the hypothesis of Conjecture \ref{conj-FPP} holds for $\pi$, then there exists a parabolic subgroup $P = MN$ of $G$ containing $P_f=M_fN_f$ such that the following holds:
        \begin{itemize}
            \item Suppose the simple factors of $M$ are $F_1$, $\dots$, $F_k$, so that the inducing module $\pi_{M}$ in Equation \eqref{eq-bl} has simple factors $\pi_{F_1}$, $\dots$, $\pi_{F_k}$, then the lowest $F_i \cap K$-types of $\pi_{F_i}$ are cx-basic ; and
            \item At least one of the $\pi_{F_i}$ satisfies the hypothesis of the FPP conjecture (Theorem \ref{thm-reduce}). 
        \end{itemize}
    This will effectively reduce the study of FPP conjecture to the cases when the lowest $K$-type of $\pi = \pi_F$ is cx-basic.

    \item[(III)] Finally, we use the result of \cite{B89}, \cite{WZ23} (for classical groups) and \cite{BC05}, \cite{BC09} (for exceptional groups) to show that $\pi_F$ is nonunitary on a certain bottom layer $(F \cap K)$-type of level $\mathfrak{p}$. For instance, in the simply laced case where all $\pi_F$ is spherical, we take the adjoint $(F \cap K)$-type. Then Theorem \ref{thm-SV} implies that $\pi$ is not unitary up to level $\mathfrak{p}$, and the conjecture follows as a consequence.
\end{enumerate}

\section{cx-basic $K$-types} \label{sec-basic}
In this section, we will give the definition of cx-basic $K$-types, and will use it to determine the Levi subgroup $M$ appearing in Section \ref{sec-strategy}. As a consequence, we will prove that step (II) in Section \ref{sec-strategy} holds.

\begin{definition} \label{def-cxfund}
    Let $G$ be a simple Lie group. We say a $K$-type $V_{\mathfrak{g}}(\tau)$ {\bf cx-basic} if $\tau$ is of the following form:
    \begin{itemize}
        \item Type $A_n$ ($n \geq 1$), $D_n$ ($n \geq 4$) or $E_n$ ($n = 6,7,8$): $\tau = [0,\dots,0]$
        \item Type $B_n$: $\tau = [0,\dots,0]$ or $[0,\dots,0,1]$.
        \item Type $C_n$: $\tau = [0,\dots,0]$ or $[0,\dots,0,1,0,\dots,0]$.
        \item Type $F_4$: $\tau = [0,0,0,0]$, $[0,0,1,0]$ or $[0,0,0,1]$.
        \item Type $G_2$: $\tau = [0,0]$, $[1,0]$ or $[2,0]$.
    \end{itemize}
\end{definition}

For every $K$-type $V_{\mathfrak{g}}(\eta)$, there exists a unique maximal Levi subgroup $M_{cx}$ containing $M_f$ such that $V_{\mathfrak{m}}(\eta)$ is cx-basic. Here are a few examples:
\begin{example}
\begin{enumerate}
    \item[(a)] If $G$ is simply laced, then one always have $M_{cx} = M_f$, since the only cx-basic $K$-type is the trivial representation. 
    \item[(b)] Suppose $G$ is of type $B_8$ and $\eta = [2,0,0,3,0,1,0,1]$, then $M_f$ is of type $A_2 \times A_1 \times A_1$ corresponding to the nodes $I(M_f) = \{2,3,5,7\}$. Meanwhile, $M_{cx} \supset M_f$ corresponds to the nodes $I(M_{cx}) = \{2,3,5,7,8\}$.
    \item[(c)] Suppose $G$ is of type $G_2$, Then $I(M_{cx}) = \phi$ if $\eta = [a,b]$ with $a,b \geq 1$; $I(M_{cx}) = \{1\}$ if $\eta = [0,b]$ with $b \geq 1$; $I(M_{cx}) = \{2\}$ if $\eta = [a,0]$ with $a \geq 3$; and $I(M_{cx}) = \{1,2\}$ if $\eta = [0,0], [1,0]$ or $[2,0]$. 
    \item[(d)] Suppose $G$ is of type $F_4$, we give a table relating $\eta$ and $I(M_{cx})$. The coordinates marked by $\bullet$ can be taken as any integer $\geq 1$, and the coordinates marked by $\ast$ can be any integer $\geq 2$.
\begin{center}
\begin{longtable}{|c|c|c|c|} 
\caption{List of $M_{cx}$ for highest weights in $F_4$} \label{table-F4}\\
\hline
$I(M_{cx})$ & $\eta$  \\
\hline \hline
$\phi$ & 
$[\bullet,\bullet,\bullet,\bullet]$\\
\hline
$\{1\}$ & 
$[0,\bullet,\bullet,\bullet]$\\
\hline
$\{2\}$ & 
$[\bullet,0,\ast,\bullet]$\\
\hline
$\{3\}$ & 
$[\bullet,\bullet,0,\bullet]$\\
\hline
$\{4\}$ & 
$[\bullet,\bullet,\bullet,0]$\\
\hline
$\{1,2\}$ & 
$[0,0,\ast,\bullet]$\\
\hline
$\{1,3\}$ & 
$[0,\bullet,0,\bullet]$\\
\hline
$\{1,4\}$ & 
$[0,\bullet,\bullet,0]$\\
\hline
$\{2,3\}$ & 
$[\bullet,0,0,\ast]$, $[\bullet,0,1,\bullet]$\\
\hline
$\{2,4\}$ & 
$[\bullet,0,\ast,0]$\\
\hline
$\{3,4\}$ & 
$[\bullet,\bullet,0,0]$\\
\hline
$\{1,2,3\}$ & 
$[0,0,0,\ast]$, $[0,0,1,\bullet]$\\
\hline
$\{1,2,4\}$ & 
$[0,0,\ast,0]$\\
\hline
$\{1,3,4\}$ & 
$[0,\bullet,0,0]$\\
\hline
$\{2,3,4\}$ & 
$[\bullet,0,0,0]$, $[\bullet,0,1,0]$, $[\bullet,0,0,1]$\\
\hline
$\{1,2,3,4\}$ & 
$[0,0,0,0]$, $[0,0,1,0]$, $[0,0,0,1]$\\
\hline
\end{longtable}
\end{center}
\end{enumerate}
\end{example}
Here is the main theorem of this section:
\begin{theorem} \label{thm-reduce}
Let $\pi$ be an irreducible, Hermitian $(\mathfrak{g},K)$-module satisfying the hypothesis of Conjecture \ref{conj-FPP}, and $M_{cx} \supset M_f$ is defined as above. Consider the $(\mathfrak{m}_{cx}, M_{cx} \cap K)$-module $\pi_{M_{cx}}$ appearing in Equation \eqref{eq-bl} with $M= M_{cx}$. Then there is a simple factor $\pi_H$ of $\pi_{M_{cx}}$ satisfying the hypothesis of Conjecture \ref{conj-FPP}.
\end{theorem}

\begin{proof}
Firstly, consider the induced module in Equation \eqref{eq-bl} with $M = M_f$. Suppose the simple factors of the inducing module $\pi_{M_f}$ are the spherical modules $\pi_{F_1}$, $\dots$, $\pi_{F_l}$. Then each $\pi_{F_i}$ is uniquely determined by its $\nu_i$-parameter. More precisely, suppose $I(F_i) = \{i_1, i_2, \dots, i_k\}$, then the infinitesimal character $\nu_i$ of $\pi_{F_i}$ is given by (a $W(F_i)$-conjugate of):
$$\nu_i = x_1 \beta_{i_1} + x_2 \beta_{i_2} + \dots + x_k \beta_{i_k}\quad \quad (x_i \in \mathbb{R}),$$
where $\beta_i$ are the simple roots of $G$. Moreover, if $\pi$ admits an invariant Hermitian form, then one must have $w_i\nu_i = -\nu_i$ for some element $w_i$ in the Weyl group corresponding to $H_i$. In such a setting, the Langlands parameter of $\pi_{M_f} = J_{M_f}(\lambda,-w\lambda)$ as well as $\pi = J(\lambda,-w\lambda)$ are given by:
\begin{equation} \label{eq-lambda}
(\lambda,-w\lambda) = \left(\frac{\eta}{2} + \nu_1 + \nu_2 + \dots + \nu_l;\ \frac{-\eta}{2} + \nu_1 + \nu_2 + \dots + \nu_l\right)
\end{equation}
with $\langle \eta, \nu_i\rangle = 0$ for all $i$. 

\smallskip
Now we are in the position to prove the theorem. Let $M_{cx} \supset M_f$ be the Levi subgroup defined in the beginning of this section, and $\pi_{H_i}$ are the simple factors of $\pi_{M_{cx}}$ in the induced module \eqref{eq-bl}. Suppose on contrary that the infinitesimal character of all $\pi_{H_i}$ lies inside FPP, then the theorem follows one can prove either of the following holds: 
\begin{enumerate}
    \item[(i)] $\pi = J(\lambda,-w\lambda)$ is not fully supported, or \item[(ii)] the infinitesimal character of $\pi = J(\lambda,-w\lambda)$ lies inside FPP. 
\end{enumerate}

Let $\alpha(\ell):= (\overbrace{a_{1}, a_{2}, \dots, -a_{2},-a_{1}}^{\ell\ \mathrm{terms}})$ for some $a_1 \geq a_2 \geq \dots \geq a_{\lfloor \ell/2 \rfloor} \geq 0$ and $\nu(\ell) = (\nu_1, \dots, \nu_{\ell})$ for $\nu_1 \geq \dots \geq \nu_l \geq 0$. For classical groups, the $\lambda$ in Equation \eqref{eq-lambda} are listed in usual coordinates as follows:

Type $A$: 
$$\lambda = (\cdots|\ \overbrace{\frac{i+1}{2}, \dots, \frac{i+1}{2}}^{r_{i+1}\ \mathrm{terms}}\ |\ \overbrace{\frac{i}{2},\dots,\frac{i}{2}}^{r_{i}\ \mathrm{terms}}\ |\ \cdots) +(\cdots \ |\ \alpha(r_{i+1})\ |\ \alpha(r_i)\ |\ \cdots)$$

Type $B$: 
\[
\lambda = \begin{cases}
(\cdots|\ \overbrace{\frac{i}{2}, \dots, \frac{i}{2}}^{r_{i}\ \mathrm{terms}}\ |\ \cdots\ |\ \overbrace{\frac{1}{2},\dots,\frac{1}{2}}^{r_{1}\ \mathrm{terms}}\ |\ \overbrace{0,\dots,0}^{r_{0}\ \mathrm{terms}}) + (\cdots\ |\  \alpha(r_{i})\ |\ \cdots\ |\ \alpha(r_{1})\ |\ \nu(r_0));\ or \\
(\cdots\ |\ \overbrace{\frac{4i+3}{4}, \dots, \frac{4i+3}{4}}^{r_{\frac{4i+3}{2}}\ \mathrm{terms}}\ |\ \overbrace{\frac{4i+1}{4}, \dots, \frac{4i+1}{4}}^{r_{\frac{4i+1}{2}}\ \mathrm{terms}}\ |\ \cdots) + (\cdots|\ \alpha(r_{\frac{4i+3}{2}})\ |\  \alpha(r_{\frac{4i+1}{2}})\ |\ \cdots)\\
\end{cases}
\]

Type $C$:
\[
\lambda = (\cdots|\ \overbrace{\frac{i}{2}, \dots, \frac{i}{2}}^{r_{i}\ \mathrm{terms}}\ |\ \cdots\ |\ \overbrace{\frac{2}{2},\dots,\frac{2}{2}}^{r_{2}\ \mathrm{terms}}\ |\ \overbrace{\frac{1}{2},\dots,\frac{1}{2}}^{r_{1}\ \mathrm{terms}}, \overbrace{0,\dots,0}^{r_{0}\ \mathrm{terms}}) + (\cdots\ |\  \alpha(r_{i})\ |\ \cdots\ |\  \alpha(r_{2})\ |\ \alpha(r_{1}); \nu(r_0))
\]

Type $D$:
\[
\lambda = \begin{cases}
(\cdots|\ \overbrace{\frac{i}{2}, \dots, \frac{i}{2}}^{r_{i}\ \mathrm{terms}}\ |\ \cdots\ |\ \overbrace{\frac{1}{2},\dots,\frac{1}{2}}^{r_{1}\ \mathrm{terms}}\ |\ \overbrace{0,\dots,0}^{r_{0}\ \mathrm{terms}}) + (\cdots\ |\  \alpha(r_{i})\ |\ \cdots\ |\ \alpha(r_{1})\ |\ \nu(r_0)); or \\
(\cdots\ |\ \overbrace{\frac{4i+3}{4}, \dots, \frac{4i+3}{4}}^{r_{\frac{4i+3}{2}}\ \mathrm{terms}}\ |\ \overbrace{\frac{4i+1}{4}, \dots, \frac{4i+1}{4}}^{r_{\frac{4i+1}{2}}\ \mathrm{terms}}\ |\ \cdots) + (\cdots|\ \alpha(r_{\frac{4i+3}{2}})\ |\  \alpha(r_{\frac{4i+1}{2}})\ |\ \cdots)\\
\end{cases}
\]

Here each term within $| \cdot |$ gives the infinitesimal character of the simple factor $\pi_{H_i}$ appearing in $\pi_{M_{cx}}$. By assumption, 
these coordinates must lie inside FPP. On the other hand, if  
$\pi$ is fully supported, then Proposition \ref{prop-full} implies that the strings of numbers separated by the $| \cdot |$'s are {\it interlaced} (c.f. \cite{DW20}, \cite{DW22}). 

Combining these two observations, one can easily show that $\pi$ is either not fully supported, or in the case when $\pi$ is fully supported ($\Leftrightarrow$ the strings of numbers are interlaced), then the dominant form $\Lambda = s\lambda$ must lie inside the FPP.

\medskip
As for exceptional cases, we begin with $G$ of exceptional type $G_2$. By our choice of simple roots in Table \ref{table-adjoint}, the simple roots $\beta_1 = [2,-1]$ and $\beta_2 = [-3,2]$.

Firstly, note that we only need to study the case when $M_{cx} \neq G$. Then all such possible $\lambda = \frac{1}{2}\eta + \nu$ are given by:
$$\lambda = \begin{cases}
    \frac{1}{2}[a,0] + x\beta_2 = \frac{1}{2}[a,0] + x\cdot[-3,2] = [\frac{a}{2}-3x,2x] & a \geq 3,\ 0 \leq x \leq  \frac{1}{2}\\
    \frac{1}{2}[0,b] + y\beta_1 = \frac{1}{2}[0,b] + y\cdot[2,-1] = [2y,\frac{b}{2}-
    y] & b \geq 1,\ 0 \leq y \leq  \frac{1}{2}\\
    \frac{1}{2}[c,d] & c,d \geq 1\\
\end{cases}$$
where $I(M_{cx}) = \{1\}$ in the first case, $I(M_{cx}) =\{2\}$ in the second case, and $I(M_{cx}) = \phi$ in the third case. 

In the case when $I(M) = \phi$, $\pi = J(\frac{1}{2}[c,d],-\frac{1}{2}[c,d]) = X(\frac{1}{2}[c,d],-\frac{1}{2}[c,d])$ is a tempered representations with empty support, so one can ignore this case. As for the first two cases, one requires 
$$0 \leq \langle x\beta_2,\beta_2^{\vee}\rangle = 2x \leq 1  \quad \quad \quad 0 \leq \langle y\beta_1,\beta_1^{\vee}\rangle = 2y \leq 1$$ 
respectively, so that the infinitesimal character of $\pi_{M_{cx}}$ lies inside the FPP. By Proposition \ref{prop-full}, $\pi$ is fully supported if  $\frac{a}{2}-3x < 0$ in the first case and $\frac{b}{2}-y < 0$ in the second case. However, this is not possible by our choices of $a$ and $b$. 

\smallskip
As for the other exceptional groups, we use computer software to finish the checking.  
To present the calculations needed behind the proof, consider the example when $G$ is of type $F_4$ and $\eta = [0,0,2,0]$, so that $I(M_{cx}) = \{1,2,4\}$. In this case, $\lambda = \frac{1}{2}\eta+\nu$ is given by:
$$\lambda = \frac{1}{2}\eta + x(\beta_1+\beta_2) + y\beta_4 = [x,x,1-2x-y,2y], \quad 0 \leq x \leq 1,\ 0 \leq y \leq \frac{1}{2}.$$
For instance, the $A_2$ factor $\pi_{A_2}$ corresponding to the nodes $1$ and $2$ has $\nu$-values chosen as a multiple of $(\beta_1 + \beta_2)$ so that $\pi_{A_2}$ admits an invariant Hermitian form. Moreover, the infinitesimal character of $\pi_{A_2}$ lies inside FPP if and only if $0 \leq \langle x(\beta_1+\beta_2),\beta_i^{\vee} \rangle\leq 1$ for $i = 1,2$, i.e. $0 \leq x \leq 1$. 

By looking at the expression of $\lambda$ above, Proposition \ref{prop-full} implies that if $\pi$ is fully supported, then $1-2x-y < 0$. Under such assumptions, one can show that for all $0 \leq x \leq 1$ and $0 \leq y \leq \frac{1}{2}$, the resulting $\Lambda = s\lambda$ lies inside FPP, violating the hypothesis of Conjecture \ref{conj-FPP}.

The auxiliary computer files are available from the following DOI:
\begin{verbatim}
 10.13140/RG.2.2.35634.00960
\end{verbatim}
\end{proof}




\section{Proof of FPP conjecture}\label{sec-proof-of-FPP}
We begin this section with a lemma:
\begin{lemma} \label{lem-bottom}
    Let $\pi$ be an irreducible, Hermitian $(\mathfrak{g},K)$-module with real infinitesimal character and lowest $K$-type $V_{\mathfrak{g}}(\eta)$. Suppose $M_{cx} \supset M_f$ is the Levi subgroup defined in the beginning of Section \ref{sec-basic}, and $\pi_{M_{cx}}$ is the inducing module in Equation \eqref{eq-bl}. Then for any simple factor $\pi_{H}$ of $\pi_{M_{cx}}$ such that $\pi_H$ is a spherical $(\mathfrak{h}, H \cap K)$-module, the adjoint $(H \cap K)$-type is $V_{\mathfrak{m}_{cx}}(\eta+\delta)$ bottom layer.
\end{lemma}
The proof follows essentially from some case-by-case calculations Example \ref{eg-bottom}. Similar analysis is also made in \cite[Proposition 2.3]{BW23}. We present an example below:
\begin{example}
    Let $G$ be of type $E_8$ with $\eta = [1,2,0,0,0,3,0,0]$. Then $M_{cx} = M_f$ with $I(M) = \{3,4,5,7,8\}$. There are two simple spherical factors $\pi_{H_1}$ (of type $A_3$, corresponding to nodes $\{3,4,5\}$) and $\pi_{H_2}$ (of type $A_2$, corresponding to nodes $\{7,8\}$). 

    The adjoint representation of $\pi_{H_1}$ has highest weight $\delta_1 = \beta_3+\beta_4+\beta_5$. Then one can check that $\langle \eta + \delta_1, \beta_i^{\vee} \rangle \geq 0$ for all simple roots $\beta_i$. More explicitly , one has
    \begin{align*}
    \langle \eta + \delta_1,  \beta_1^{\vee} \rangle &= 1 + \langle \beta_3, \beta_1^{\vee} \rangle = 1 - 1 \geq 0\\
    \langle \eta + \delta_1,  \beta_2^{\vee} \rangle &= 2 + \langle \beta_4, \beta_2^{\vee} \rangle = 2 - 1 \geq 0\\
    \langle \eta + \delta_1,  \beta_6^{\vee} \rangle &= 3 + \langle \beta_5, \beta_2^{\vee} \rangle = 3 - 1 \geq 0
    \end{align*}
    and $\langle \eta + \delta_1, \beta_j^{\vee} \rangle = \langle \eta, \beta_j^{\vee} \rangle \geq 0$ for $j \neq 1,2,6$.

    Similarly, the adjoint representation of $\pi_{H_2}$ has highest weight $\delta_2 = \beta_7+\beta_8$. Then $\langle \eta + \delta_2, \beta_k^{\vee} \rangle = \langle \eta, \beta_k^{\vee} \rangle \geq 0$ for $k \neq 6$, and 
    $$\langle \eta + \delta_1, \beta_6^{\vee} \rangle = 3 + \langle \beta_7, \beta_6^{\vee} \rangle  = 3 -1 \geq 0.$$
\end{example}

Now we are in the position to prove Conjecture \ref{conj-FPP}:

\smallskip
\noindent {\it Proof of Conjecture \ref{conj-FPP}.} Retain the settings in Lemma \ref{lem-bottom}, and suppose $\pi = J(\lambda,-w\lambda)$ satisfies the hypothesis of Conjecture \ref{conj-FPP}. By Theorem \ref{thm-reduce}, there must be a simple factor $\pi_H$ of $\pi_{M}$ satisfying the hypothesis of Conjecture \ref{conj-FPP}.

If $\pi_H$ is spherical, then the results in \cite[Lemma 3.3]{BC05}, \cite[Lemma 3.2]{BC09} imply that it has indefinite signature on the (level $\mathfrak{p}$) adjoint $(H \cap K)$-type . As a consequence, Lemma \ref{lem-bottom} immediately implies that this $(M \cap K)$-type is bottom layer, and hence it has indefinite signature on $\pi$ as well.

So we are left with the following cases of non-spherical simple factors $\pi_H$ of $\pi_M$ satisfying the hypothesis of Conjecture \ref{conj-FPP}:

\medskip
\noindent \underline{Type $B$:} $M_{cx}$ of type $A_{i_1} \times \dots \times A_{i_k} \times B_l$, and $\pi_H = \pi_{B_l}$ has lowest $(H \cap K)$-type $V_{\mathfrak{b}_l}(\tau)$, where $\tau = [0,\dots,0,1]$.

\smallskip
\noindent \underline{Type $C$:}
$M_{cx}$ of type $A_{i_1} \times \dots \times A_{i_k} \times C_l$, and $\pi_H = \pi$ has lowest $(H \cap K)$-type $V_{\mathfrak{c}_l}(\tau)$, where $\tau = [0,\dots,0,1,0,\dots,0]$.

\smallskip
\noindent \underline{Type $F_4$:}
\begin{itemize}
    \item $I(M_{cx}) = \{2,3\}$ of type $B_2$, and $\pi_H = \pi_{B_2}$ has lowest $(H \cap K)$-type $V_{\mathfrak{b}_2}(\tau)$, where $\tau = [0,1]$.
    \item $I(M_{cx}) = \{1,2,3\}$ of type $B_3$, and $\pi_H = \pi_{B_3}$ has lowest $(H \cap K)$-type $V_{\mathfrak{b}_3}(\tau)$, where $\tau = [0,0,1]$.
    \item $I(M_{cx}) = \{2,3,4\}$ of type $C_3$, and $\pi_H = \pi_{C_3}$ has lowest $(H \cap K)$-type $V_{\mathfrak{c}_3}(\tau)$, where $\tau = [0,1,0]$ or $[0,0,1]$ (in $F_4$ coordinates).
    \item $I(M_{cx}) = \{1,2,3,4\}$ of type $F_4$, and $\pi_H = \pi_{F_4}$ has lowest $(H \cap K)$-type $V_{\mathfrak{f}_4}(\tau)$, where $\tau = [0,0,1,0]$ or $[0,0,0,1]$.
    \end{itemize}

\noindent \underline{Type $G_2$:}
$I(M_{cx}) = \{1,2\}$ of type $G_2$, and $\pi_H = \pi_{G_2}$ has lowest $(H \cap K)$-type $V_{\mathfrak{g}_2}(\tau)$, where $\tau = [1,0]$ or $[2,0]$.    

We now deal with these cases separately:
\begin{lemma} \label{lem-b}
    Let $G$ be of type $B_l$, and $\pi$ be a Hermitian, irreducible $(\mathfrak{g},K)$-module with real infinitesimal character and lowest $K$-type $V_{\mathfrak{b}_l}(\tau)$, where 
    $$\tau = [0,\dots,0,1]\quad \quad (i.e.\ \tau = (\frac{1}{2},\frac{1}{2},\dots,\frac{1}{2})\ \text{in usual coordinates}).$$ 
    Suppose $\pi_{B_l}$ satisfies they hypothesis of Conjecture \ref{conj-FPP}, then its Hermitian form has indefinite signature on the level $\mathfrak{p}$ $K$-type $V_{\mathfrak{b}_l}(\gamma)$ with
$$\gamma = \tau + \beta_1+\beta_2 + \dots +\beta_n$$    
(in usual coordinates, $\gamma = (\frac{3}{2},\frac{1}{2},\dots,\frac{1}{2})$).
\end{lemma}
\begin{proof}
    For convenience, we work on the usual coordinates, where $\frac{\tau}{2} = (\frac{1}{4},\dots,\frac{1}{4})$, and $\pi_{B_l} = J(\lambda,-w\lambda)$ with 
    $$\lambda = \frac{\tau}{2} + \nu = (\frac{1}{4},\dots,\frac{1}{4}) + \alpha(l) = (\frac{1}{4}+\nu_1,\frac{1}{4}+\nu_2,\dots,\frac{1}{4}-\nu_2,\frac{1}{4}-\nu_1).$$
    Partition the above coordinates of $\lambda$ into $$\lambda =   \lambda' \sqcup \lambda_{\frac{1}{2}} \quad \quad -w\lambda =   -w_A\lambda' \sqcup -w_B\lambda_{\frac{1}{2}}, $$ 
    where $\lambda_{\frac{1}{2}}$ contains all coordinates of $\lambda$ of the form $\frac{1}{2}\mathbb{Z}$, and $\lambda'$ are the remaining coordinates. Then the Kazhdan-Lusztig conjecture implies that $\pi_{B_l}$ is of the form:
    $$\pi = \mathrm{Ind}_{LU}^G(J_{A_{l-p-1}}(\lambda',-w_{A}\lambda') \boxtimes J_{B_p}(\lambda_{\frac{1}{2}},-w_{B}\lambda_{\frac{1}{2}})\boxtimes 1),$$ 
    where $L$ is the Levi subgroup of $G$ of type $A_{l-p-1} \times B_p$ with $p$ being the number of coordinates of $\lambda_{\frac{1}{2}}$. 

    By hypothesis, the dominant form of $\lambda$ does not lie inside FPP. Then it is easy to check that the dominant form of either $\lambda'$ or $\lambda_{\frac{1}{2}}$ also does not lie inside FPP. In the first case, $J_{A_{l-p-1}}(\lambda',-w_{A}\lambda')$ has opposite signatures on the $\widetilde{U}(l-p)$-types with highest weights $(\frac{1}{2},\frac{1}{2},\dots, \frac{1}{2},\frac{1}{2})$ and $(\frac{3}{2},\frac{1}{2},\dots, \frac{1}{2},\frac{-1}{2})$.
    In the second case, one can either apply (a small generalization of) \cite[Section 4.5]{BDW22} or simply \cite[Theorem 7.6]{WZ23} to conclude that $(\lambda_{\frac{1}{2}},-w_{B}\lambda_{\frac{1}{2}})$ has opposite signatures on the $Spin(2p+1)$-types with highest weights $(\frac{1}{2},\frac{1}{2},\dots, \frac{1}{2},\frac{1}{2})$ and $(\frac{3}{2},\frac{1}{2},\dots, \frac{1}{2})$. In both cases, the result follows from the preservation of indefiniteness of Hermitian forms under real parabolic induction (c.f. \cite[Theorem 10.5]{V86}).
\end{proof}

\begin{lemma} \label{lem-c}
    Let $G$ be of type $C_l$, and $\pi$ be a Hermitian, irreducible $(\mathfrak{g},K)$-module with real infinitesimal character and lowest $K$-type $V_{\mathfrak{c}_l}(\tau)$, where 
    $$\tau = [\overbrace{0,\dots,0}^{(i-1)\ \mathrm{terms}},1,\overbrace{0,\dots,0}^{j\ \mathrm{terms}}]\quad \quad (i.e.\ \tau = (\overbrace{1,\dots,1}^{i\ \mathrm{terms}},\overbrace{0,\dots,0}^{j\ \mathrm{terms}})\ \text{in usual coordinates}).$$ 
    Suppose $\pi$ satisfies they hypothesis of Conjecture \ref{conj-FPP}, then its Hermitian form has indefinite signature on the level $\mathfrak{p}$ $K$-type $V_{\mathfrak{c}_l}(\gamma_i)$ ($i = 1,2$) with
$$\gamma_1 = \tau + \beta_1+ \dots + \beta_i, \quad \quad  \gamma_2 =
\tau + \beta_1+ \dots + \beta_{i+1} + 2(\beta_{i+2} + \dots + \beta_{i+j-1}) +\beta_{i+j}.$$
(in usual coordinates, $ \gamma_1 = (2,\overbrace{1\dots,1}^{(i-2)\ \mathrm{terms}},\overbrace{0,\dots,0}^{(j+1)\ \mathrm{terms}})$, $\gamma_2 = (2,\overbrace{1\dots,1}^{i\ \mathrm{terms}},\overbrace{0,\dots,0}^{(j-1)\ \mathrm{terms}})$). 
\end{lemma}
\begin{proof}
    Firstly, note that the Levi subgroup $M_f$ attached to $\pi$ is of type $A_{i-1} \times C_j$. Let 
    $$\pi_{M_f} := \pi_A \boxtimes \pi_C$$
    be the inducing module in \eqref{eq-bl} with $M = M_f$. By hypothesis and Theorem \ref{thm-reduce}, either the $\frac{1}{2}\eta + \nu$ parameter on $\pi_A$ or the $\frac{1}{2}\eta + \nu$ parameter on $\pi_C$ does not lie inside FPP.
    In the first case, the result follows from Lemma \ref{lem-bottom} that $\pi$ has indefinite form on $V_{\mathfrak{c}_l}(\gamma_1)$. 

    \medskip
     In the second case, we adopt the main argument used in \cite[Chapter 9]{B89}. Firstly, we assume $\pi_A$ is unitary up to level $\mathfrak{a}_{i-1} \cap \mathfrak{p}$, otherwise the $\pi$ is not unitary at $V_{\mathfrak{c}_l}(\gamma_1)$ as in the first case. Write $\pi = J(\lambda,-w\lambda)$ with
    $$\lambda = \frac{\tau}{2} + \nu = (\overbrace{\frac{1}{2},\dots,\frac{1}{2}}^{i\ \mathrm{terms}};\overbrace{0,\dots,0}^{j\ \mathrm{terms}}) + (\alpha(i);\nu(j)) = (\lambda_1,\dots,\lambda_i; \lambda_{i+1},\dots,\lambda_{i+j}).$$
    By hypothesis, neither the dominant form of $\lambda$ nor the dominant form of $(\lambda_{i+1},\dots,\lambda_{i+j})$
    lies inside FPP. Therefore, there must be some $1 \leq r \leq j$ such that
    \begin{equation} \label{eq-ineq}
    \lambda_{i+1}, \dots, \lambda_{i+r} > \mathrm{max}\{|\lambda_1|,\dots,|\lambda_i|,|\lambda_{i+r+1}|, \dots, |\lambda_{i+j}|\}+1.
    \end{equation}
    
    For all $x \geq 0$, let 
    $$\pi_C(x) := J((\lambda_{i+1}+x,\dots,\lambda_{i+r}+x,\lambda_{i+r+1},\dots,\lambda_{i+j}),(\lambda_{i+1}+x,\dots,\lambda_{i+r}+x,\lambda_{i+r+1},\dots,\lambda_{i+j})).$$ 
    be a spherical module, and  
    $$I(x) := \mathrm{Ind}_{M_fN_f}^{Sp(2l)}(\pi_A \boxtimes \pi_C(x) \boxtimes 1).$$ 
    Denote its lowest $K$-type subquotient is denoted as $\pi(x)$. By Equation \eqref{eq-ineq} and some standard analysis on intertwining operators, the multiplicities and signatures of the $K$-type $V_{\mathfrak{c}_l}(\gamma_2)$ in $I(x)$ 
    is equal to that of $\pi(x)$ for all $x \geq 0$.

    \smallskip
    Suppose on contrary that $\pi = \pi(0)$ has definite signature at the $K$-type $V_{\mathfrak{c}_l}(\gamma_2)$, the above arguments imply the same conclusion holds for $I(x)$ for all $x \geq 0$.
    Now take $x = x_0$ be the smallest non-negative number such that $M:=\lambda_{i+r}+x_0 \in \frac{1}{2}\mathbb{N}$ and 
    \begin{equation} \label{eq-M}
    M > \max\{|\lambda_1|,\dots,|\lambda_i|\} +2 \end{equation}
    We will show that $I(x_0)$ has indefinite form on 
    $V_{\mathfrak{c}_l}(\gamma_2)$.
    
    Consider the trivial module $F(M-1) := J_{A_{2m-2}}((M-1,M-2,\dots,-(M-1));(M-1,M-2,\dots,-(M-1)))$ of $GL(2M-1)$, and the induced module
    \begin{equation} \label{eq-indc}
    \begin{aligned}
    &\mathrm{Ind}_{GL(2M-1) \times Sp(2l)}^{Sp(2l+4M-2)}(F(M-1) \boxtimes I(x_0)) \\ 
    =\ &\mathrm{Ind}_{GL(2M-1) \times GL(i) \times Sp(2j)}^{Sp(2l+4M-2)}(F(M-1) \boxtimes \pi_A \boxtimes \pi_C(x))\\ 
    =\ &\mathrm{Ind}_{GL(i) \times GL(2M-1) \times Sp(2j)}^{Sp(2l+4M-2)}(\pi_A \boxtimes F(M-1) \boxtimes \pi_C(x)),
    \end{aligned}
    \end{equation}
    Since we assume $I(x_0)$ has definite form on $V_{\mathfrak{c}_l}(\gamma_2)$, the induced module \eqref{eq-indc} also has definite signature on the $K$-type $V_{\mathfrak{c}_{l+2M-1}}((2,\overbrace{1,\dots,1}^{i\ \mathrm{terms}},\overbrace{0,\dots,0}^{(j+2M-2)\ \mathrm{terms}}))$. On the other hand, the arguments in \cite[Lemma 8.2(1)]{B89} implies that a composition factor $\Xi$ of the induced module \eqref{eq-indc} is the lowest $K$-type subquotient of
    $\mathrm{Ind}_{GL(i+2)\times Sp(2j+4M-6)}^{Sp(2j+4M-2)}(\pi_A' \boxtimes \pi_C')$, where 
    \begin{align*}
        \pi_A' &:= J\left((\frac{1}{2},\dots,\frac{1}{2}) + \alpha(i), M, -(M-1)); (\frac{-1}{2},\dots,\frac{-1}{2}) + \alpha(i), M-1, -M)\right)\\
        \pi_C' &:= J(\xi,\xi),\quad \xi := (\lambda_{i+1}+x_0,\dots,\lambda_{i+r-1}+x_0,M-1,M-2,\dots,-(M-2)).
    \end{align*}
    In particular, its lowest $K$-type is $V_{\mathfrak{c}_{l+2M-1}}(\tau')$, where
    $$\tau' = (\overbrace{1,\dots,1}^{(i+2)\ \mathrm{terms}},\overbrace{0,\dots,0}^{(j+2M-2)\ \mathrm{terms}}).$$
    But the inequality \eqref{eq-M} and the validity of Conjecture \ref{conj-FPP} for Lie groups of type $A$ imply that $\pi_A'$ has indefinite Hermitian form on the $K$-type $V_{\mathfrak{a}_{i+1}}((2,\overbrace{1,\dots,1}^{i\ \mathrm{terms}},0))$, which is bottom layer in $\Xi$. Consequently, $\Xi$ also has indefinite form on $V_{\mathfrak{c}_{l+2M-1}}((2,\overbrace{1,\dots,1}^{i\ \mathrm{terms}},\overbrace{0,\dots,0}^{(j+2M-1)\ \mathrm{terms}}))$, contradicting the fact that all composition factors of \eqref{eq-indc} have definite Hermitian form on $V_{\mathfrak{c}_{l+2M-1}}((2,\overbrace{1,\dots,1}^{i\ \mathrm{terms}},\overbrace{0,\dots,0}^{(j+2M-1)\ \mathrm{terms}}))$.
\end{proof}

One can easily check that the $\gamma$'s given in Lemma \ref{lem-b} and Lemma \ref{lem-c} are always $M_{cx}$-bottom layer for any $G$ and $\eta$. For instance, if $G$ is of type $F_4$ $\eta = [0,0,1,\bullet]$ with $\bullet \geq 1$. Then Table \ref{table-F4} implies that $I(M_{cx}) = \{1,2,3\}$ with a single simple factor $H$ of type $B_3$. By Lemma \ref{lem-b}, the indefinite $(H \cap K)$-type $V_{\mathfrak{h}}(\gamma) = V_{\mathfrak{h}}(\eta + \beta_1 + \beta_2 + \beta_3)$ satisfies 
\begin{align*}
\langle \eta + \beta_1 + \beta_2 + \beta_3, \beta_1^{\vee} \rangle &= 0 + 2 - 1 + 0\geq 0 \\
\langle \eta + \beta_1 + \beta_2 + \beta_3, \beta_2^{\vee} \rangle &= 0 - 1 + 2 - 1 \geq 0 \\
\langle \eta + \beta_1 + \beta_2 + \beta_3, \beta_3^{\vee} \rangle &= 1 + 0 - 2 + 2 \geq 0 \\
\langle \eta + \beta_1 + \beta_2 + \beta_3, \beta_4^{\vee} \rangle &= \bullet + 0 + 0 -1 \geq 0 
\end{align*}
and hence it is bottom layer.

\medskip
Consequently, we are left with the exceptional nonspherical cases, that is: 

\begin{itemize}
    \item Type $F_4$: $\pi_{F_4}$ has lowest $K$-type $V_{\mathfrak{f}_4}(\tau)$, where $\tau = [0,0,1,0]$ or $[0,0,0,1]$; and
    \item Type $G_2$: $\pi_{G_2}$ has lowest $K$-type $V_{\mathfrak{g}_2}(\tau)$, where $\tau = [1,0]$ or $[2,0]$. 
\end{itemize} 

For each $\tau$ listed above, the Langlands parameters $(\lambda,-w\lambda) = (\frac{\tau}{2} + \nu, \frac{-\tau}{2} + \nu)$ such that the irreducible module $\pi = J(\lambda,-w\lambda)$ has lowest $K$-type $V_{\mathfrak{g}}(\tau)$ are listed below:
\begin{itemize}
    \item Type $F_4$: 
    \begin{align*}
    \tau = [0,0,1,0]&: \left([x,x,\frac{1}{2}-2x-z,2z],\ [x,x,-\frac{1}{2}-2x-z,2z]\right), \quad x,z \geq 0 \\
    \tau = [0,0,0,1]&:\ \left([a,b,c,\frac{1}{2}-(a+2b+\frac{3}{2}c)],\ [a,b,c,-\frac{1}{2}-(a+2b+\frac{3}{2}c)]\right), \quad a,b,c \geq 0. 
    \end{align*}
    
    \item Type $G_2$: 
        \begin{align*}
    \tau = [1,0]&: \left([\frac{1}{2}-3x,2x],\ [-\frac{1}{2}-3x,2x]\right), \quad x \geq 0 \\
    \tau = [2,0]&:\ \left([1-3x,2x],\ [-1-3x,2x]\right), \quad x \geq 0.
    \end{align*}
\end{itemize} 

Consider the induced module $I(\lambda) := \mathrm{Ind}_{M_fN_f}^G(\pi_{M_f} \boxtimes 1)$ in Equation \eqref{eq-bl}. Then the character formula of $\pi = J(\lambda,-w\lambda)$ can be split into two parts:
\begin{equation}\label{eq-twoparts}
\begin{aligned}
\pi = \sum_{w_1 \in W(M_f)} \alpha_{w_1}(\lambda) X(\lambda,-w_1w\lambda)
+ \sum_{w_2 \notin W(M_f)} \alpha_{w_2}(\lambda)X(\lambda,-w_2w\lambda)
\end{aligned}
\end{equation}
where $\alpha_{w_1}(\lambda), \alpha_{w_2}(\lambda) \in \mathbb{Z}$, and we set $\alpha_{w'} = 0$ whenever $\lambda - w'w\lambda$ is not in the weight lattice to avoid any ambiguity (c.f. Section \ref{sec-cx}).

Note the first part of \eqref{eq-twoparts} is the character formula of $I(\lambda)$, and 
$\alpha_{w'}(\lambda) \neq 0$ implies that the dominant form $\lambda - w'w\lambda$ must be equal to $\eta$ plus some {\it positive} roots of $G$. 

Suppose $\pi$ satisfies the hypothesis of Conjecture \ref{conj-FPP}. Then Theorem \ref{thm-reduce} implies that the same hypothesis holds for $\pi_{M_f}$. Consequently, $\pi_{M_f}$ must have indefinite signature on some level $\mathfrak{p}$ $(M_f \cap K)$-type by Lemma \ref{lem-bottom}. Then by \cite[Theorem 10.5]{V86} on signatures of real parabolically induced module again, $I(\lambda)$ would have indefinite signature on the level $\mathfrak{p}$ $K$-type $V_{\mathfrak{g}}(\gamma)$ given by:
\begin{itemize}    
\item Type $F_4$: 
\begin{align*}
\eta = [0,0,1,0]&:\ \gamma = [1,0,0,1]\quad or \quad [0,0,0,2];\\  
\eta = [0,0,0,1]&:\ \gamma = [0,0,1,0] 
\end{align*}

\item Type $G_2$: 
\begin{align*}
\eta = [1,0]&:\ \gamma = [2,0];\\  
\eta = [2,0]&:\ \gamma = [1,1] 
\end{align*}
\end{itemize}

We make the following {\it claim}.

\medskip
Let $(\lambda,-w\lambda)$ be as given above, such that the dominant form of $\lambda$ lies outside of FPP. If there exists $w_2 \notin W(M_f)$ such 
that the $\lambda - w_2w\lambda$ is equal to $\eta$ plus some positive roots. Then the 
dominant form $\widetilde{\gamma}$ of $\lambda - w_2w\lambda$ has \texttt{height} higher than the $\gamma$'s listed above.

\medskip
Assuming the validity of the claim, then the
second part of Equation \eqref{eq-twoparts} 
does not contain any $K$-type of the form $V_{\mathfrak{g}}(\gamma)$. Hence the claim implies that $I(\lambda)$ and $\pi$ have the same multiplicity and signature on the $K$-type $V_{\mathfrak{g}}(\gamma)$, verifying the conjecture in these cases.

\medskip
Let us prove the claim.

For $G_2$ and $\eta = [1,0]$, 
$$(\lambda,-w\lambda) = \left([\frac{1}{2}-3x,2x],\ [-\frac{1}{2}-3x,2x]\right)$$ 
satisfies the hypothesis if and only if $x > \frac{3}{2}$. Indeed, for any $w\in W(G_2)$, let $w\lambda=[t_1, t_2]$. Assuming that $w\lambda$ is dominant, a direct check shows that $\max\{t_1, t_2\}\leq 1$ can not happen if $x>\frac{3}{2}$; moreover,   $\max\{t_1, t_2\}>1$ can not happen if $0\leq x\leq \frac{3}{2}$. Under the assumption that $x>\frac{3}{2}$, we further check that  
the second part of \eqref{eq-twoparts} does not contain
$[2,0]$.

For $G_2$ and $\eta = [2,0]$, 
$$(\lambda,-w\lambda) = \left([1-3x,2x],\ [-1-3x,2x]\right)$$ satisfies the hypothesis if and only if $x > \frac{2}{3}$. Indeed, for any $w\in W(G_2)$, let $w\lambda=[t_1, t_2]$. Assuming that $w\lambda$ is dominant, a direct check shows that $\max\{t_1, t_2\}\leq 1$ can not happen if $x>\frac{2}{3}$; moreover,   $\max\{t_1, t_2\}>1$ can not happen if $0\leq x\leq \frac{2}{3}$. Under the assumption that $x>\frac{2}{3}$, we further check that  
the second part of \eqref{eq-twoparts} does not contain
$[1,1]$.

For $F_4$ and $\eta = [0,0,1,0]$, 
$$(\lambda,-w\lambda) = \left([x,x,\frac{1}{2}-2x-z,2z],\ [x,x, -\frac{1}{2}-2x-z,2z]\right).$$ 
If $z > \frac{1}{2}$, then Lemma \ref{lem-c} would immediately imply that the module has indefinite form on $V_{\mathfrak{f}_4}([0,0,0,2])$. So we assume $0 \leq z \leq \frac{1}{2}$ from now on. Then one can easily check the hypothesis is satisfied if and only if $x > 2$. Indeed, for any $w\in W(F_4)$, let $w\lambda=[t_1, t_2, t_3, t_4]$. Assuming that $w\lambda$ is dominant, a direct check shows that $\max\{t_1, t_2, t_3, t_4\}\leq 1$ can not happen if $x>2$; moreover,   $\max\{t_1, t_2, t_3, t_4\}>1$ can not happen if $0\leq x\leq 2$. Under the assumption that $x>2$, we further check that  
the second part of \eqref{eq-twoparts} does not contain
$[1,0,0,1]$ or $[0,0,0,2]$.

\smallskip
For $F_4$ and $\eta = [0,0,0,1]$, 
$$(\lambda,-w\lambda) =\left([a,b,c,\frac{1}{2}-(a+2b+\frac{3}{2}c)],\ [a,b,c,-\frac{1}{2}-(a+2b+\frac{3}{2}c)]\right).
$$ 

We enumerate that there are $29$ elements $w_1$ in the Weyl group of $F_4$ such that there exists $a,b,c$ such that $\max\{a,b,c\}>1$ and that $\Lambda = \{\lambda\}=w_1\lambda$ is outside FPP.  For example, when $w_1=s_1 s_2 s_3 s_4$ we have that $\Lambda=w_1\lambda= [b+\frac{c}{2}-\frac{1}{2}, a,-a-\frac{c}{2}+ \frac{1}{2}, c]$. When $\max\{a,b,c\}>1$, one computes that $\Lambda$ is outside FPP if and only if the following holds:
\begin{itemize}
\item[$\bullet$] $0\leq a <\frac{1}{2}$, $1+a<b\leq \frac{3}{2}$, $3-2 b<c\leq 1-2 a$; or 
\item[$\bullet$] $0\leq a <\frac{1}{2}$, $\frac{3}{2}<b$, $0\leq c\leq 1-2 a$; or
\item[$\bullet$]  $a=\frac{1}{2}$, $\frac{3}{2}<b$, $c=0$.
\end{itemize}
Under the above assumptions, we check the second part of \eqref{eq-twoparts} does not contain $[0,0,1,0]$. We handle the other $28$ $w_1$'s similarly.

The details for $F_4$ are contained in the file ‘‘F4-FPP-two-etas.pdf" under the aforementioned DOI in Section \ref{sec-proof-of-FPP}.

\bigskip
This completes the proof of Conjecture \ref{conj-FPP}. \qed

\medskip

\centerline{\scshape Funding} 
Dong is supported by the National Natural Science Foundation of China (no. 12171344).
Wong is supported by the National Natural Science Foundation of China (no. 12371033) and the 
Shenzhen Science and Technology Innovation Committee grant
(no. 20220818094918001).

\centerline{\scshape Acknowledgements} 
We wish to thank the Tianyuan Mathematical Center in Southeast China and Tianyuan Mathematics Research Center for offering wonderful working conditions so that the authors can meet and finish this paper.

\end{document}